\pdfoutput=1
\documentclass{amsart}
\usepackage{amsmath}
\usepackage{graphicx}

\newcommand{\bigcircle}{\begin{picture}(10,10)
\put(5,2){\circle{10}}
\end{picture}
}

\newcommand{\x}{~\begin{picture}(10,15)
\put(5,4){\circle{15}}
\put(0,-1){\vector(1,1){10}}
\put(10,-1){\vector(-1,1){10}}
\end{picture}~
}

\newcommand{\dxu}{~
\begin{picture}(10,15)
\put(5,11.5){\circle*{3}}
\put(5,4){\circle{15}}
\put(0,-1){\vector(1,1){10}}
\put(10,-1){\vector(-1,1){10}}
\end{picture}~
}

\newcommand{\dxl}{~
\begin{picture}(10,15)
\put(5,4){\circle{15}}
\put(10,-1){\vector(-1,1){10}}
\put(0,-1){\vector(1,1){10}}
\put(-2,4){\circle*{3}}
\end{picture}~
}
\newcommand{\dxr}{~
\begin{picture}(10,15)
\put(5,4){\circle{15}}
\put(0,-1){\vector(1,1){10}}
\put(10,-1){\vector(-1,1){10}}
\put(12,4){\circle*{3}}
\end{picture}~
}
\newcommand{\dxb}{~
\begin{picture}(10,15)
\put(5,4){\circle{15}}
\put(0,-1){\vector(1,1){10}}
\put(10,-1){\vector(-1,1){10}}
\put(5,-3.5){\circle*{3}}
\end{picture}~
}

\theoremstyle{plain}
\newtheorem{theorem}{Theorem}
\newtheorem{lemma}{Lemma}
\newtheorem{proposition}{Proposition}
\newtheorem{corollary}{Corollary}

\theoremstyle{definition}
\newtheorem{notation}{Notation}
\newtheorem{definition}{Definition}

\theoremstyle{remark}
\newtheorem{remark}{Remark}

\begin{document}
\title{Based chord diagrams of spherical curves
}
\author{Noboru Ito}
\address{Graduate School of Mathematical Sciences, ICMS iBMath, The University of Tokyo, 3-8-1, Komaba, Meguro-ku, Tokyo 153-8914, Japan.}
\thanks{
The work was partly supported by a Waseda University Grant for Special Research Projects (Project number: 2014K-6292) and The JSPS Japanese-German Graduate Externship.  The author was a project researcher of Grant-in-Aid for Scientific Research (S) (No.~24224002) (2016.4--2017.3).  The work was partly supported by Sumitomo Foundation (Grant for Basic Science Research Projects, Project number: 160556).}
\email{noboru@ms.u-tokyo.ac.jp}
\maketitle

\begin{abstract}
This paper demonstrates an approach for developing a framework to produce invariants of base-point-free generic spherical curves under some chosen local moves from Reidemeister moves using based chord diagrams.  Our invariants not only contain Arnold's classical generic spherical curve invariant but also new invariants.      
\end{abstract}
\section{Introduction}
A {\it{spherical curve}} is the image of a generic immersion of a circle into a two-dimensional sphere.  
Any two spherical curves are related by a finite sequence consisting of three types of local replacements, namely, {\it{Reidemeister moves}} RI, RI\!I, and RI\!I\!I (Fig.\ \ref{f1}).  
\begin{figure}[h!]
\centering
\includegraphics[width=8cm]{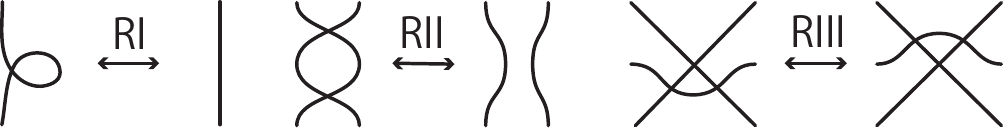}
\caption{Reidemeister moves.}\label{f1}
\end{figure}
Moves RI\!I and RI\!I\!I can be decomposed into two types of moves: strong and weak moves (Fig.\ \ref{f2}).    
\begin{figure}[h!]
\centering
\includegraphics[width=12cm]{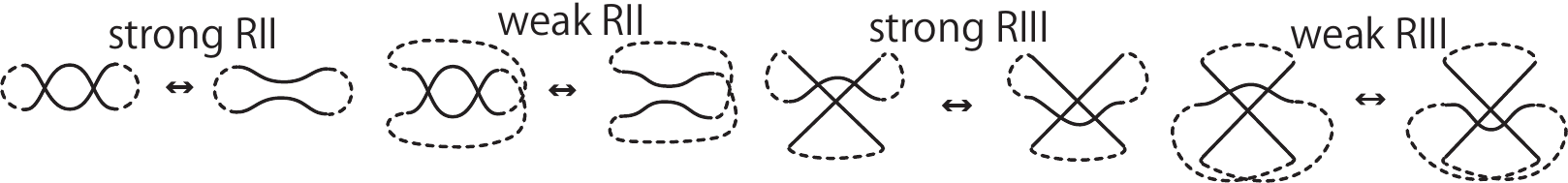}
\caption{Reidemeister moves strong RI\!I, weak RI\!I, strong RI\!I\!I, and weak RI\!I\!I.  Dotted arcs indicate the connections of the branches.}\label{f2}
\end{figure}

A {\it{chord diagram of a spherical curve}} is the immersing oriented circle on which the preimages of all the double points are placed and connected by a chord (Definition~\ref{d2}).  
For each double point of a spherical curve, a unique replacement can be obtained by orienting the spherical curve arbitrarily, as shown in Fig.\ \ref{f3}.   
Here, this replacement does not depend on the orientation of the spherical curve.
% because the spherical curve is one component.  
\begin{figure}[h!]
\centering
\includegraphics[width=3cm]{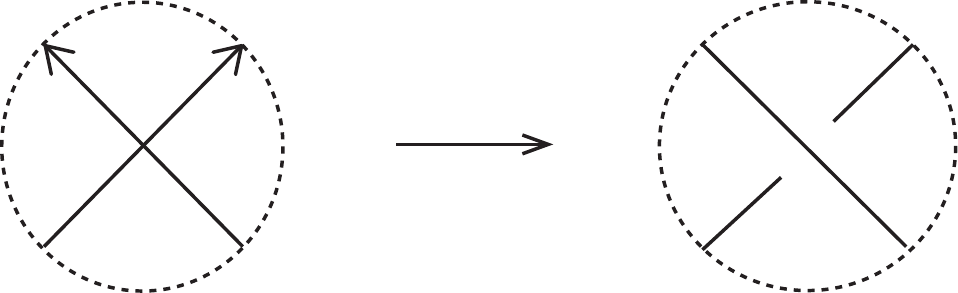}
\caption{Resolution from an oriented double point to an unoriented crossing.}\label{f3}
\end{figure} 
A spherical curve with over- /under- information of every double point is said to be a {\it{knot diagram}}.  A double point of a knot diagram is called a {\it{crossing}}.  A crossing as shown in the right figure of Fig.\ \ref{f3} is called a {\it{negative crossing}}.    
The chord diagram of every knot diagram consists of oriented chords; this is known as 
 an {\it{arrow diagram}}, previously introduced by Polyak and Viro \cite{PV} (see also \cite{polyak}).  An oriented chord is called an {\it{arrow}}.  
\begin{figure}[h!]
\centering
\includegraphics[width=10cm]{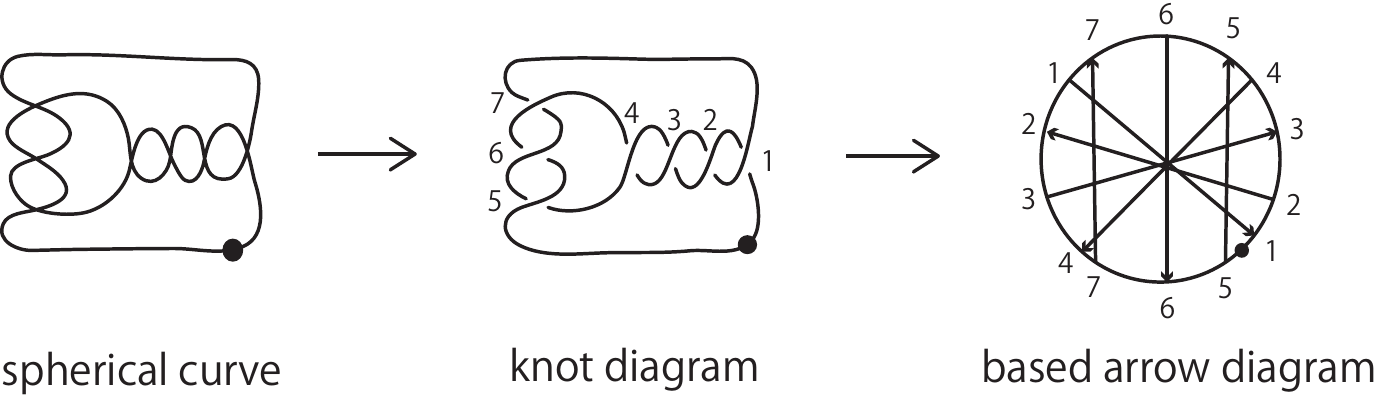}
\caption{Spherical curve with a base point, its knot diagram with a base point, and its based arrow diagram whose circle is oriented counterclockwise.}\label{f4}
\end{figure}

An arrow diagram with a base point is called a {\it{based arrow diagram}} (Definition~\ref{d3}).  Fig.\ \ref{f4} describes the process for obtaining a based arrow diagram from a spherical curve with a base point.  
% up to orientation preserving homeomorphism.  
As shown in Fig.\ \ref{f4}, we select an arbitrary base point on the spherical curve that does not coincide with any of the double points.  
%Then we have an arrow diagram with a base point, called a {\it{based arrow diagram}}.  

We are interested in such based arrow diagrams for the following reasons.  If we count the number of sub-arrow diagrams of type $\x$ embedded into the entire (based) arrow diagram of a spherical curve $P$, the result is an integer determined by $P$, denoted as $\x(P)$.  We can see that $\x(P)$ (mod $3$) is a $\mathbb{Z}/3\mathbb{Z}$-valued invariant under RI and strong RI\!I\!I \cite{ITT} and $\x(P)$ (mod $4$) is a $\mathbb{Z}/4\mathbb{Z}$-valued invariant under RI and strong RI\!I \cite{IT_some_ch}.   Here, when we consider invariants under certain Reidemeister moves, as a $\mathbb{Z}$-valued quantity, $\x(P)$ produces only an invariant under RI.  

By contrast, sub-based arrow diagrams can be used to define a greater number of invariants of spherical curves without a base point than a single invariant $\x(P)$ (Theorem~\ref{thm1}).  

The following is the plan of this paper.  Sec.\ \ref{sec1h} contains definitions, notations and one of the main results.    
Sec.\ \ref{sec2} presents the proof of Theorem~\ref{thm1} and Sec.\ \ref{sec3} explains the framework that produces these invariants.  Sec.\ \ref{secArnold} shows that one of the invariants in Theorem~\ref{thm1} is equal to a linear combination of Arnold's invariants.  {Sec.~\ref{sec4}} introduces a new invariant under RI and weak RI\!I because an invariant under RI and weak RI\!I is not covered by Theorem~\ref{thm1}.    
Finally, in Sec.\ \ref{secTable}, we construct tables from the values of these invariants for prime knot projections without $1$-gons up to seven double points.

\section{Definitions, notations, and main results}\label{sec1h}
\begin{definition}[oriented Gauss word]\label{d1}
Let $\hat{n}=\{1, 2, \dots, n \}$.  A \emph{word} $w$ of length $n$ is a map from $\hat{n}$ to $\mathbb{N}$ and each element of $w(\hat{n})$ is called a \emph{letter}.  Traditionally, the word is represented by $w(1)w(2) \cdots w(n)$.   A \emph{Gauss word} of length $2n$ is a word $w$ of length $2n$ such that each letter appears exactly twice in $w(1)w(2) \cdots w(2n)$.   
For a given Gauss word $w$ and for each letter $k$, we distinguish the two $k$'s in $w$ by calling one $k$ a \emph{head} and the other \emph{tail}.  The assignments are expressed by adding extra information to $w$ $=$ $w(1)w(2) \cdots w(2n)$, that is, we add ``$\bar{~}$" on the letters which are assigned tails.  This new word $w^*$ is called an \emph{oriented Gauss word}.   We call each letter of an oriented Gauss word an \emph{oriented letter}.  
Without loss of generality, we may suppose that the set of the letters in $w(\hat{2n})$ is $\{ 1, 2, \ldots, n \}$.  Clearly, the set of oriented letters of the word $w^*$ of length $2n$, denoted by $w^*(\hat{2n})$, is $\{ 1, 2, \ldots n, \bar{1}, \bar{2}, \ldots, \bar{n} \}$.  Let $v^*$ be another oriented Gauss word that is induced from $v$.    
Two oriented Gauss words, $v^*$ and $w^*$, of length $2n$ are \emph{isomorphic} if there exists a bijection $f : w(\hat{2n}) \to v(\hat{2n})$ such that $v^*$ $=$ $f^* \circ w^*$ where $f^* : w^{*}(\hat{2n})$ $= \{ 1, 2, \ldots, n, \bar{1}, \bar{2}, \ldots, \bar{n} \}$ $\to$ $v^*(\hat{2n})$ is the bijection such that $f^*(i)$ $=$ $f(i)$ and $f^*(\bar{i})$ $=$ $\overline{f(i)}$ ($i$ $=$ $1, 2, \ldots, n$).  Isomorphism induces an equivalence relation on oriented Gauss words.  For an oriented Gauss word $w$ of length $2n$, the equivalence class containing $w^*$ is denoted by $[w]$.  
\end{definition}
\begin{definition}[chord diagram, arrow diagram, and based arrow diagram]\label{d2}
A \emph{chord diagram} is a configuration of $2n$ paired points on an oriented circle.   The two points of each pair are usually connected by a straight arc, called a \emph{chord}.  An \emph{arrow diagram} is a chord diagram such that each pair of points consists of a starting point and an end point and the circle is oriented counterclockwise.  The orientation of each chord is represented by an arrow from the tail to the head.  Each oriented chord is called an \emph{arrow}.   A \emph{based arrow diagram} is an arrow diagram with a base point on the circle where the base point does not coincide with one of the paired points.   
\end{definition}
Note that we have one to one correspondence between the equivalence classes of Gauss words of length $2n$ and the based arrow diagrams, each of which has $n$ arrows, as shown in Fig.\ \ref{3h}.  In the rest of this paper, we identify these four expressions in Fig.\ \ref{3h}, and freely use this identification.  
\begin{figure}[h!]
\includegraphics[width=10cm]{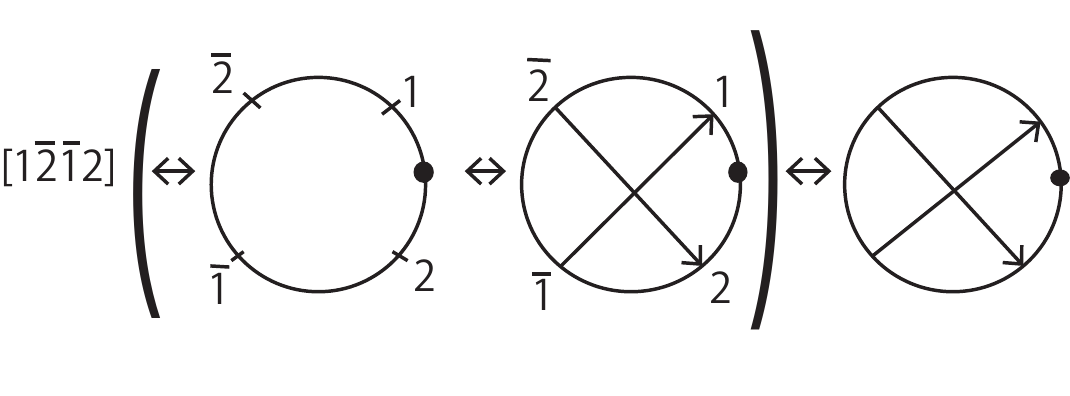}
\caption{Four expressions.}\label{3h}
\end{figure}
\begin{definition}[a based arrow diagram $AD_P$ of a spherical curve $P$]\label{d3}
Let $P$ be a spherical curve.  By definition, there exists a generic immersion $g : S^1$ $\to$ $S^2$ such that $g(S^1)=P$.  We define a base chord diagram of $P$ (e.g., Fig.\ \ref{f4}) as follows.  Let $l$ be the number of the double points of $P$.  Fix a base point, which does not coincide with a double point, and choose an orientation of $P$.  The spherical curve with the orientation and the base point is denoted by $\dot{P}^+$ and the oriented spherical curve $P$ with the opposite orientation having the base point is denoted by $\dot{P}^-$.  

Starting at the base point, proceed along $\dot{P}^+$ according to the orientation of $\dot{P}^+$.  To begin with, we assign integer $1$ to the first double point that we encounter.  Then we assign integer $2$ to the next double point that we encounter provided it is not the first double point which has been already assigned integer $1$.    If we have already assigned $1$, $2$, \ldots, $p$ and we encounter the next double point that has not assigned yet, then we assign $p+1$ to it.  Following the same procedure, we finish to label all the double points of $\dot{P}^+$.   
By definition, $g^{-1}(i)$ consists of two points on $S^1$ and we shall assign $i$ to them.  
Then the chord diagram with a base point is represented by $g^{-1}($the base point on $P$), $g^{-1}($double point assigned $1)$, $g^{-1}($double point assigned $2)$, \ldots, $g^{-1}($double point assigned $l)$ on a circle.   
Then, $g^{-1}($the base point on $P$) on $S^1$ is called the \emph{base point} of a chord diagram.       

Next, we consider the knot diagram obtained from $\dot{P}^+$ by replacing every double point with a negative crossing with respect to the orientation, as shown in Fig.\ \ref{f3}.    We assign an orientation to each chord where the head corresponding to the under path.  Then, this based arrow diagram is denoted by $AD_{\dot{P}^+}$ and is called a \emph{based arrow diagram} of $P$.  
\end{definition}
Note that by definition, it is easy to show that the based arrow diagram $AD_{\dot{P}^+}$ is the reflection image of $AD_{\dot{P}^-}$.   Note also that the based arrow diagram $AD_{\dot{P}^{\epsilon}}$ gives an equivalence class of oriented Gauss words, say $[v_{\dot{P}^\epsilon}]$.  Then, by the definition of the equivalence relation, it is elementary to show that the map $\dot{P}^\epsilon$ $\mapsto$ $[v_{\dot{P}^\epsilon}]$ is well-defined.  
\begin{notation}[$x(P)$]\label{n1}
 Let $x$ be a based arrow diagram.  For a given spherical curve $P$, we choose and fix a base point and an orientation of $P$, i.e., we obtain $\dot{P}^\epsilon$ ($\epsilon = +$ or $-$).  We fix an oriented Gauss word $G$ isomorphic to $w_{\dot{P}^\epsilon}$.   
 We consider the set $\{ G' ~|~$ $G'$ is obtained from $G$ by ignoring some pairs of oriented letters $\}$.  Then, we consider the subset of this set consisting of the elements, each of which is isomorphic to $x$.   We denote the cardinality of this subset by $x(G)$.  Suppose that $H$ is another oriented Gauss word isomorphic to $w_{\dot{P}^\epsilon}$.  By definition, $x(G)=x(H)$.  Thus, this number, determined by $[w_{\dot{P}^\epsilon}]$, is denoted by $x(\dot{P}^\epsilon)$.  If $P$ is oriented and has a base point, i.e., $P$ $=$ $\dot{P}^{\epsilon}$ but $x(\dot{P}^\epsilon)$ does not depend on the choice of the base point and the orientation, we may write $x(P)$ to represent $x(\dot{P}^\epsilon)$.   
 \end{notation}
 \begin{definition}[connected sum, additivity]\label{d4}
Let $P_i$ be spherical curve ($i=1, 2$).  Suppose that the ambient $2$-spheres are oriented.  
Let $p_i$ be a point on $P_i$ such that $p_i$ does not coincide with a double point ($i=1, 2$).  Let $d_i$ be a sufficiently small disk with center $p_i$ ($i=1, 2$) where $d_i \cap P_i$ consists of an arc properly embedded in $d_i$.  Let $\hat{d}_i$ $=$ $cl(S^2 \setminus d_i)$ and $\hat{P}_i$ $=$ $P_i \cap \hat{d}_i$.  Let $h :$ $\partial \hat{d}_i \to \partial \hat{d}_2$ be a homeomorphism such that $h(\partial \hat{P}_1)$ $=$ $\partial \hat{P}_2$.  Then, $\hat{P}_1 \cup_h \hat{P}_2$ obtains a spherical curve in the $2$-sphere $\hat{d}_1 \cup_h \hat{d}_2$.  The spherical curve $\hat{P}_1 \cup_h \hat{P}_2$ in the $2$-sphere is denoted by $P_1 \sharp_{(p_1, p_2, h)} P_2$ and is called a \emph{connected sum} of the spherical curves $P_1$ and $P_2$ at the pair ($p_1, p_2$) (see Fig.\ \ref{f4a}).  Let $I$ be a function on the set of the spherical curves.  We say that $I$ is \emph{additive} if $I(P_1 \sharp_{(p_1, p_2, h)} P_2)$ $=$ $I(P_1)$ $+$ $I(P_2)$ for any spherical curves $P_1$ and $P_2$, for any pair ($p_1, p_2$) consisting of points, and for any $h$.         
 \end{definition}
\begin{figure}[h!]
\centering
\includegraphics[width=12cm]{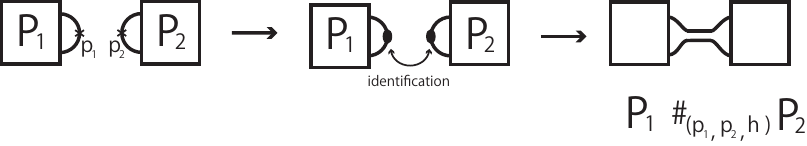}
\caption{A connected sum of two spherical curves $P_1$ and $P_2$.}\label{f4a}
\end{figure}
 
We have Theorem~\ref{thm1} as follows.
\begin{theorem}\label{thm1}
Let $P$ be a spherical curve and let $\dot{P}^\epsilon$ be $P$ with a base point and an orientation.  The integers $\dxu(\dot{P}^\epsilon)$, $\dxb(\dot{P}^\epsilon)$, and $\dxl(\dot{P}^\epsilon) + \dxr(\dot{P}^\epsilon)$ do not depend on the choice of the base point and the orientation.  Hence, these integers can be denoted by $\dxu(P)$, $\dxl(P) + \dxr(P)$, and $\dxb(P)$, respectively.  
Further, 
\[\dxl(P)+\dxr(P)-\dxb(P)\] is invariant under RI and strong RI\!I\!I,
\[\dxu(P) - \dxl(P) - \dxr(P) + \dxb(P)\] is invariant under RI and strong RI\!I, and 
\[\dxu(P) ~( = \dxb(P))\] is invariant under RI and weak RI\!I\!I.  
All these invariants are additive and non-trivial. 
\end{theorem}

\section{Proof of Theorem~\ref{thm1}}\label{sec2}
\begin{proof}    
\noindent({\bf{Independence of double points.}})   

To begin with, we show that $\dxu(\dot{P}^\epsilon)$, $\dxb(\dot{P}^\epsilon)$, and $\dxl(\dot{P}^\epsilon)+\dxr(\dot{P}^\epsilon)$ of a spherical curve $P$ are independent of the base points.  It is sufficient to consider two cases to move the base point over an arrow $A$ (Fig.\ \ref{f5}).
\begin{figure}[h!]
\centering
\includegraphics[width=5cm]{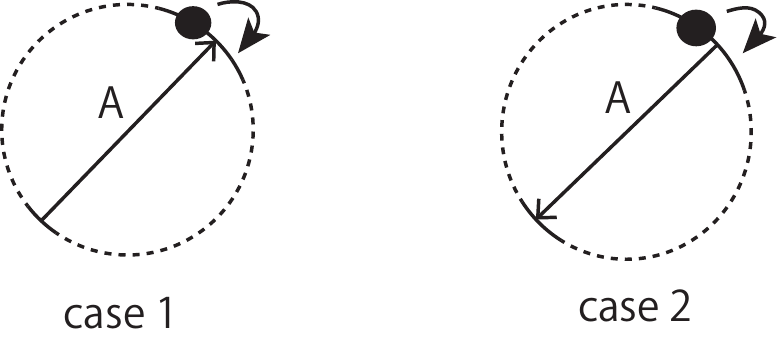}
\caption{Cases 1 and 2.}\label{f5}
\end{figure}
Before we consider the two cases, we present Lemma~\ref{lemA}.
\begin{lemma}\label{lemA}
Let $X$ be an arbitrary arrow of a based arrow diagram of an oriented spherical curve with a base point, as shown in Fig.\ \ref{f6}.  Consider the case where we move along $X$ in the direction from its start to end.  
If the number of arrows, each of which intersects the arrow $X$ from left to right is $\alpha$, then the number of arrows, each of which crosses $X$ from right to left is also $\alpha$.
\end{lemma}

\noindent({\bf{Proof of Lemma~\ref{lemA}.}})
Arrow $X$ corresponds to a double point as shown in Fig.\ \ref{f6}.  
\begin{figure}[h!]
\centering
\includegraphics[width=4cm]{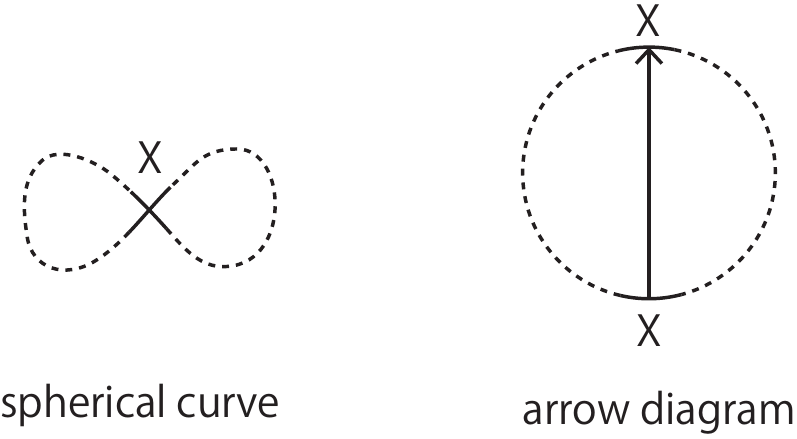}
\caption{Arrow $X$ corresponding to a double point.}\label{f6}
\end{figure}
First, choose orientation of $P$.  
Second, an operation, as shown in Fig.\ \ref{f7}, is applied to the double point that corresponds to $X$.  It is clear that the two-component spherical curves $Q$ and $R$ mutually intersect at even number of double points.
\begin{figure}[h!]
\centering
\includegraphics[width=4cm]{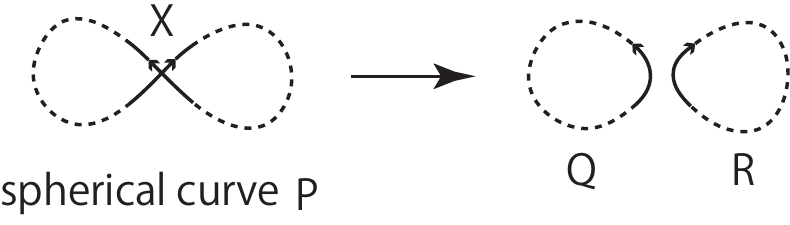}
\caption{Smoothing at the double point corresponding to $X$.  Dotted arcs indicate the connections of the spherical curves.}\label{f7}
\end{figure}
Let $d$ be an arbitrary double point of an intersection between $Q$ and $R$.  
Let $d_Q$ and $d_R$ be the branches corresponding to a double point $d$.  For $d$, there are two types of intersections, (a) and (b), as shown in the left-most column of Fig.\ \ref{f8}.
\begin{figure}[h!]
\centering
\includegraphics[width=8cm]{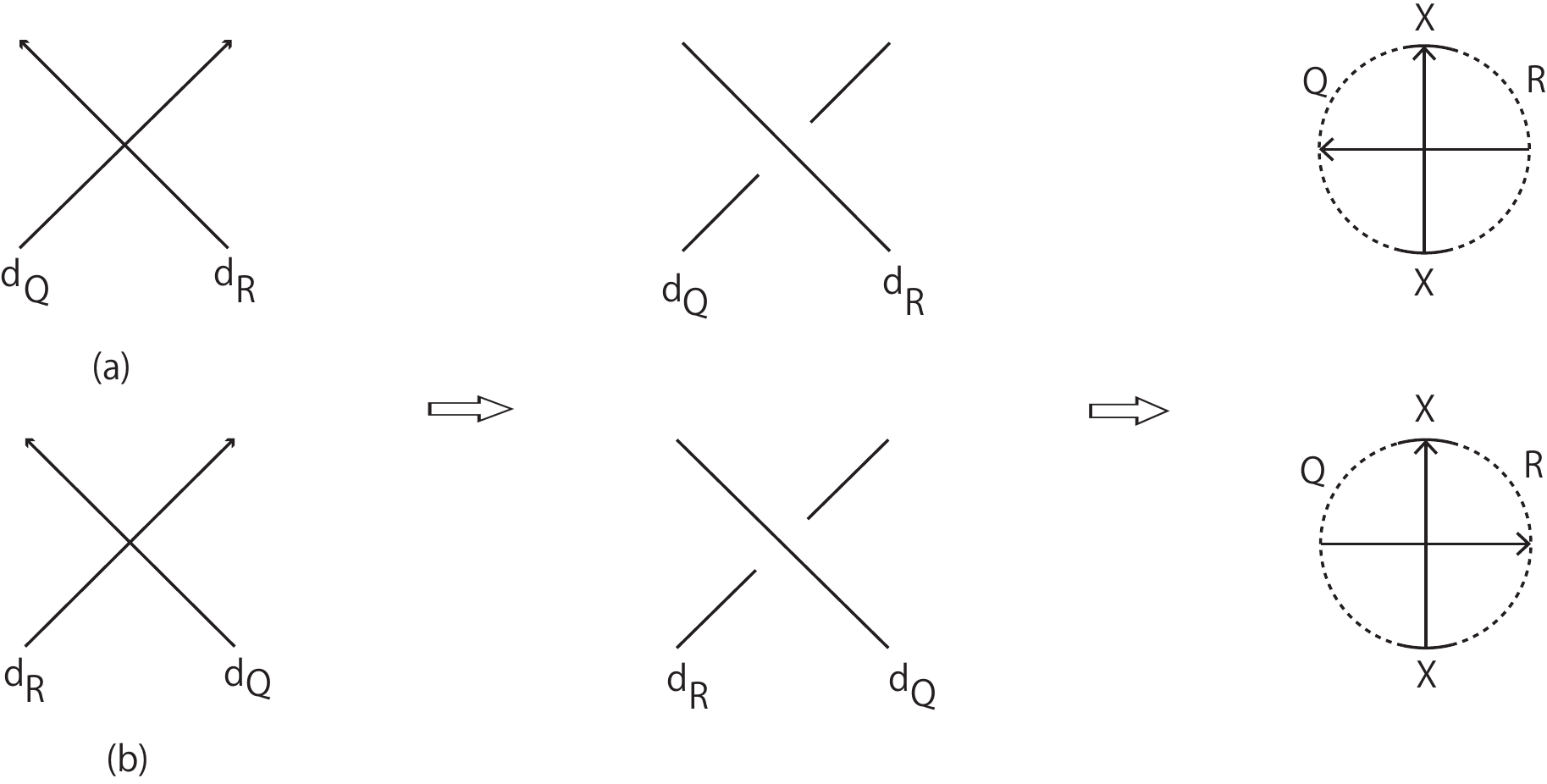}
\caption{Two types of intersections of two components $Q$ and $R$.}\label{f8}
\end{figure}
From Fig.\ \ref{f8}, it is easy to see that the number of double points of type (a) is equal to that of (b) (e.g., this can be seen when all the self-intersections of the components $Q$ and $R$ are smoothened by Seifert resolutions, as shown in the first line of {Fig.~ \ref{f8b}}.  Then, the intersections between $Q$ and $R$ become double points among simple closed curves.  Note that any two simple closed curves $q$ and $r$ mutually intersect at even number of double points.  cf.~Fig.\ \ref{f8b}).  
\begin{figure}[h!]
\begin{center}
\includegraphics[width=12cm]{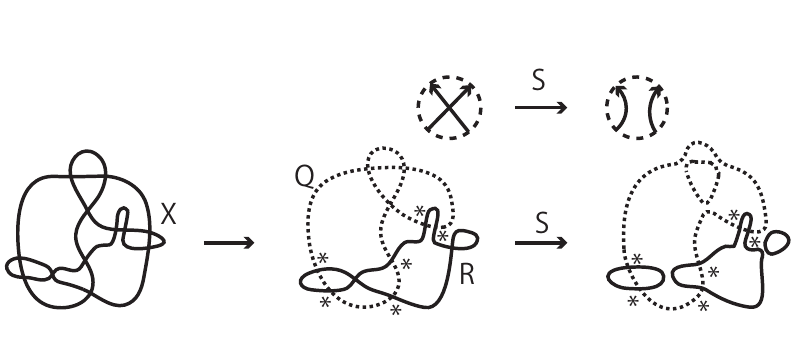}
\end{center}
\caption{$Q$ and $R$ mutually intersect at even number of double points.}\label{f8b}
\end{figure}

%(question)%%

\noindent({\bf{End of Proof of Lemma~\ref{lemA}.}})

Now we consider Cases 1 and 2 shown in Fig.\ \ref{f5} using Lemma~\ref{lemA}.
For Case~1, let $S_1$ $=$ $\{a$$~|~~\dxu$ consists of arrow $a$ and $A$ before the move is applied$\}$, and for Case~2, let $S_2$ $=$ $\{a$$~|~~\dxb$ consists of arrow $a$ and $A$ before the move is applied$\}$.  
For each case, the increment and decrement under each move are presented in Table~\ref{t1}, where $\beta_1$ $=$ $|S_1|$ and $\beta_2$ $=$ $|S_2|$.  
\begin{table}[h!]
\caption{Cases 1 and 2.}\label{t1}
\centering
\begin{tabular}{|c|c|c|} \hline
&Case 1&Case 2 \\ \hline
$\dxu(\dot{P}^\epsilon)$ & $\beta_1 \to \beta_1$ & $0 \to 0$ \\ \hline
$\dxl(\dot{P}^\epsilon)$ & $\beta_1 \to 0$ & $0 \to \beta_2$ \\ \hline
$\dxr(\dot{P}^\epsilon)$ & $0 \to \beta_1$ & $\beta_2 \to 0$ \\ \hline
$\dxb(\dot{P}^\epsilon)$ & $0 \to 0$ & $\beta_2 \to \beta_2$ \\ \hline
%Increment & $\alpha$ & $0$ & $\alpha$ & $0$ \\ \hline
%Decrement & $\alpha$ & $\alpha$ & $0$ & $0$ \\ \hline
\end{tabular}
\end{table}

\noindent({\bf{End of Proof of the independence of base points.}})

\noindent({\bf{Independence of orientations.}})  
%To obtain a based arrow diagram, we need to choose an orientation of a spherical curve with a base point.  
Next, we check the independence of the orientation for functions $\dxu(\dot{P}^\epsilon)$, $\dxl(\dot{P}^\epsilon)$, and $\dxl(\dot{P}^\epsilon)$ $+$ $\dxr(\dot{P}^\epsilon)$.  First, by definition, $\dxu(\dot{P}^\epsilon)$ does not depend on the orientation.  This is because the mirror image of the based arrow diagram $\dxu$ is $\dxu$ itself.  Similarly, it can be seen that $\dxb(\dot{P}^\epsilon)$ and $\dxl(\dot{P}^\epsilon)$ $+$ $\dxr(\dot{P}^\epsilon)$ do not depend on the orientation.    

\noindent({\bf{End of Proof of the independence of orientations.}})

In the rest of this proof, we may write $\dxu(P)$, $\dxl(P)+\dxr(P)$, $\dxb(P)$ to represent $\dxu(\dot{P}^\epsilon)$, $\dxl(\dot{P}^\epsilon)$ $+$ $\dxr(\dot{P}^\epsilon)$, and $\dxb(\dot{P}^\epsilon)$, respectively.  

\noindent({\bf{Invariance of}} $\dxl(P)+\dxr(P)-\dxb(P)$ {\bf{under RI and strong RI\!I\!I.}})  To begin with, we show the invariance of $\dxl(P)+\dxr(P)-\dxb(P)$ under RI and strong RI\!I\!I for an arbitrary spherical curve $P$.  
Note that $\dxu(P)$, $\dxl(P) + \dxr(P)$, and $\dxb(P)$ do not depend on the position of a base point; thus, the base point can be arbitrarily positioned.  
By definition, we see that for an arbitrary spherical curve, $\dxb(P)$, $\dxl(P) + \dxr(P)$, and $\dxu(P)$ are invariant under RI (Fig.\ \ref{f9}).  
\begin{figure}[h!]
\centering
\includegraphics[width=4cm]{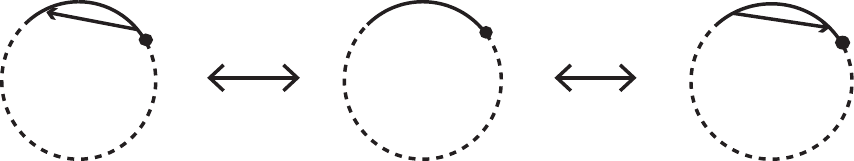}
\caption{Difference under a single RI.}\label{f9}
\end{figure}
Next, we see the differences for $\dxl(P) + \dxr(P)$ and $\dxb(P)$ with respect to a single strong RI\!I\!I, as shown in Fig.\ \ref{f10}.  A single strong RI\!I\!I increasing (resp.~decreasing) the number $\dxu(P)+\dxl(P)+\dxr(P)+\dxb(P)$ is denoted by $s3a$ (resp.~$s3b$).  
%%question%%
\begin{figure}[h!]
\centering
\includegraphics[width=4cm]{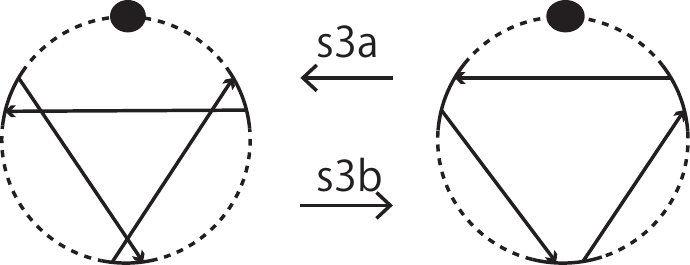}
\caption{Difference under a single strong RI\!I\!I.}\label{f10}
\end{figure}
\begin{table}[h!]
\caption{Increment under a single $s3a$ for each type.
}\label{t3}
\centering
\begin{tabular}{|c|c|} \hline
$\dxu(P)$ & $1$ \\ \hline
$\dxl(P) + \dxr(P)$ & $1$ \\ \hline
%\dxr & $0$ \\ \hline
$\dxb(P)$ & $1$ \\ \hline
\end{tabular}
\end{table}
Table~\ref{t3} provides the claim for the invariance (the row corresponding to $\dxu$ is not used here; we use it later).

\noindent({\bf{Invariance of $\dxu(P)$ $-$ $\dxl(P)$ $-$ $\dxr(P)$ $+$ $\dxb(P)$ under RI and strong RI\!I.}})
We already have the invariance of RI from the proof of the invariance of $\dxl(P) + \dxr(P) -\dxb(P)$ under RI.  We have also already shown the independence of the base point for $\dxu(P)$, $\dxl(P) + \dxr(P)$, and $\dxb(P)$.  Therefore, here, it is sufficient to show the invariance of strong RI\!I when a base point is positioned at a suitable location.  Thus, we consider the difference with respect to strong RI\!I as shown in Fig.\ \ref{f15} (note that the base point is placed at 
% the most neighbored place which is applied by strong RI\!I
a position closest to where a strong RI\!I is applied, as shown in Fig.\ \ref{f15}).  
\begin{figure}[h!]
\centering
\includegraphics[width=5cm]{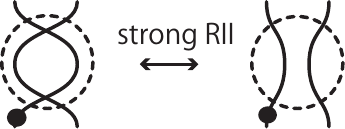}
\caption{A base point at the most neighbored place.}\label{f15}
\end{figure}
Here, strong RI\!I increasing (resp.~decreasing) double points is denoted by $s2a$ (resp.~$s2b$) as shown in Fig.\ \ref{f12}.  
\begin{figure}[h!]
\centering
\includegraphics[width=5cm]{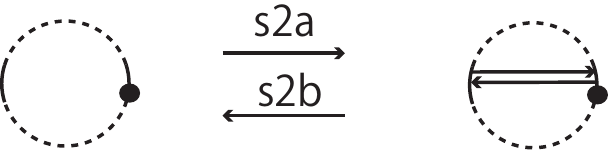}
\caption{Difference under a single strong RI\!I.}\label{f12}
\end{figure}
When we apply $s2a$ to an arbitrary spherical curve $P$ with a base point, each increment of the values of $\dxu$, $\dxl$, $\dxr$, and $\dxb$ for Cases 1 and 2 (see Fig.\ \ref{f11}) of the based arrow diagram of $P$ is shown in Table~\ref{t4}.  By Lemma~\ref{lemA}, there exist $m$ arrows shown by the dotted arrow for Case 2 if there exist $m$ arrows shown by the dotted arrow for Case 1 (Fig.\ \ref{f11}).
\begin{figure}[h!]
\centering
\includegraphics[width=5cm]{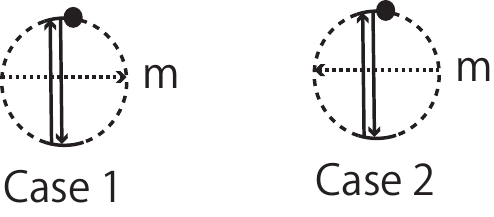}
\caption{We set the number of arrows, indicated by the dotted arrow, as $m$ (cf.~Lemma~\ref{lemA}).}\label{f11}
\end{figure}
\begin{table}[h!]
\caption{Increments under a single $s2a$.}\label{t4}
\centering
\begin{tabular}{|c|c|c|} \hline
&Case 1& Case 2 \\ \hline
$\dxu(P)$ &$m$&$0$ \\ \hline
$\dxl(P) + \dxr(P)$ &$m$&$m$ \\ \hline
%$\dxr$ &$0$&$m$ \\ \hline
$\dxb(P)$ &$0$&$m$ \\ \hline
\end{tabular}
\end{table}
Table~\ref{t4} shows the invariance of strong RI\!I.

\noindent ({\bf{Invariance of}} $\dxu(P)$ ($=\dxb(P)$) {\bf{under weak RI\!I\!I.}}) 
We have already shown the invariance under RI (at the beginning of the proof of invariance of $\dxl(P)$ $+$ $\dxr(P)$ $-$ $\dxb(P)$).  Thus, it is sufficient to show the invariance under weak RI\!I\!I.  
%%One--Single(question)%%%
A single weak RI\!I\!I is shown in Fig.\ \ref{f13} in terms of arrow diagrams.  A single weak RI\!I\!I increasing (resp.~decreasing) the number $\dxu(P)$ $+$ $\dxl(P)$ $+$ $\dxr(P)$ $+$ $\dxb(P)$ is denoted by $w3a$ (resp.~$w3b$).  
\begin{figure}[h!]
\centering
\includegraphics[width=5cm]{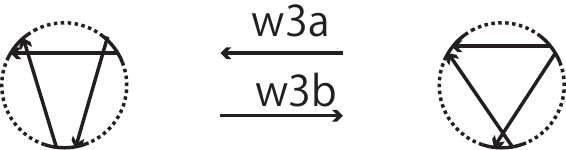}
\caption{Difference under a single weak RI\!I\!I.}\label{f13}
\end{figure}
By the independence of the base point for $\dxu(P)$ and $\dxb(P)$, we can choose a suitable position for the base point (Fig.\ \ref{f14}).  
\begin{figure}[h!]
\centering
\includegraphics[width=5cm]{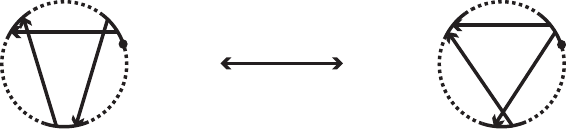}
\caption{A pair of weak RI\!I\!I with base points.}\label{f14}
\end{figure}
Table~\ref{t5} shows the increment of $\dxu(P)$ (information corresponding to $\dxl(P) + \dxr(P)$ and $\dxb(P)$ is not used here, we used it later),  
\begin{table}[h!]
\caption{Increment under a single $w3a$.}\label{t5}
\centering
\begin{tabular}{|c|c|} \hline
$\dxu(P)$ &$0$ \\ \hline
$\dxl(P) + \dxr(P)$ &$1$\\ \hline
%$\dxr$ &$0$\\ \hline
$\dxb(P)$ &$0$\\ \hline
\end{tabular}
\end{table}
which shows the invariance of $\dxu(P)$ under weak RI\!I\!I.  

\noindent ({\bf{Proof of the equality $\dxu(P)=\dxb(P)$.}})

Here, we show the equality $\dxu(P)=\dxb(P)$ for an arbitrary spherical curve $P$.  
Tables~\ref{t3} and \ref{t5} indicate that $\dxu(P)-\dxb(P)$ does not change under any RI\!I\!I.  We also know that $\dxu(P) - \dxb(P)$ is invariant under RI.  

Now, we consider its behavior under strong RI\!I.  Recall that $\dxu(P)$ and $\dxb(P)$ do not depend on the position of the base point.  
Next, we observe the difference before and after the application of a single strong RI\!I to the based arrow diagrams derived from spherical curves, as shown Fig.\ \ref{f16a}.  
\begin{figure}[h!]
\centering
\includegraphics[width=5cm]{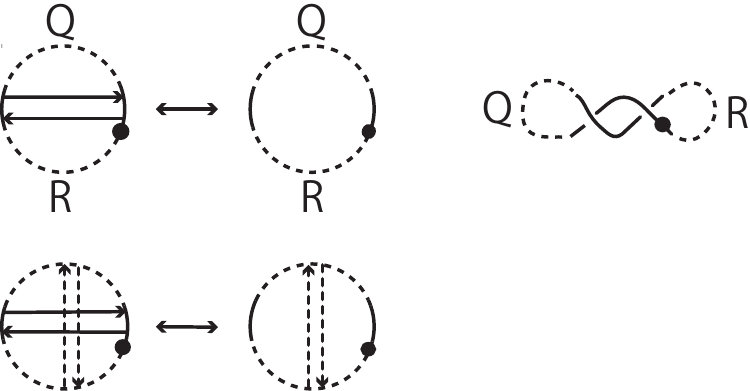}
\caption{Differences of strong RI\!I to the based arrow diagrams (left) derived from spherical curves (right).  }\label{f16a}
\end{figure}
By Lemma~\ref{lemA}, in Fig.\ \ref{f16a}, the dotted arcs $Q$ and $R$ must intersect at even number of double points.  Further, in the corresponding based arrow diagram, the number of arrows from the $Q$-part to $R$-part is equal to the number of arrows from the $R$-part to $Q$-part.  See the second row in Fig.\ \ref{f16a}.  The dotted arrows, whether mutually intersecting or not, do not contribute to the increment or decrement of $\dxu(P)$ $-$ $\dxb(P)$ under a single strong RI\!I.   
%Thus, we consider  $\dxu$ or $\dxb$, each of which consists of one dotted arrow and real arrow, as shown in the lower line of Fig.\ \ref{f16}. 
Therefore, $\dxu(P)$ $-$ $\dxb(P)$ is invariant under strong RI\!I.  

As a result, $\dxu(P)-\dxb(P)$ is invariant under RI, RI\!I\!I, and strong RI\!I.  It is easy to see that a single weak RI\!I is generated by RI, RI\!I\!I, and strong RI\!I.  Thus, $\dxu(P)-\dxb(P)$ is invariant under RI, RI\!I, and RI\!I\!I.  Therefore, $\dxu(P)-\dxb(P)$ is constant for all spherical curves.  Note that $\dxu(\bigcircle)-\dxb(\bigcircle)$ $=$ $0$ where $\bigcircle$ is a simple closed curve.  Thus, we have the equality $\dxu(P)$ $=$ $\dxb(P)$ for an arbitrary spherical curve $P$.

\noindent ({\bf{Additivity.}})

Finally, we can see that these invariants are additive because it is clear that $x(P \sharp_{(p_1, p_2)} P')$ $=$ $x(P) + x(P')$ for $x$ $=$ $\dxu$, $\dxl$, $\dxr$, or $\dxb$ where $P \sharp_{(p_1, p_2)} P'$ is the connected sum of $P$ and $P'$ (Definition~\ref{d4}).

\noindent({\bf{Non-triviality.}})
Non-trivialities of the invariants in the statement of Theorem~\ref{thm1} are shown in Tables~\ref{tt1}, \ref{tt2}, and \ref{tt4}.
\end{proof}
\section{An explanation of a framework that produces invariants of Theorem~\ref{thm1}}\label{sec3}
The following two questions (A) and (B) are essentially different.  

\noindent(A) How do we prove the invariance under some Reidemeister moves for functions in Theorem~\ref{thm1}? 

\noindent(B) How do we find functions in Theorem~\ref{thm1}?

An answer to (A) (resp.~(B)) is given in Sec.\ \ref{sec2} (resp. this section).  
Throughout this section, the following steps are applied.  Note that we implicitly consider a $\mathbb{Q}$-vector space generated by arrow diagrams, and the moves of the base point shown in Fig.\ \ref{f5} (Cases 1 and 2) are called {\it{base point moves}}.  Let $n$ be the number of observed elements in $\{ \dxu, \dxl, \dxr, \dxb, \dxl + \dxr \}$ for observed Reidemeister moves.     
\begin{itemize}
\item Step 1: Construct a matrix consisting of $n$ rows, where each row corresponds to an observed based arrow diagram and each column corresponds to an observed local move.  
\item Step 2: Compute the rank $r$ of the matrix.  Note that there exist $n-r$ invariants.
\item Step 3: Obtain the invariants using the matrix.  
\end{itemize}

%We have already proved that $\dxu(P)=\dxb(P)$ for any spherical curve $P$.  Thus, it is natural to exclude $\dxb(P)$ in this section.  
Recall that for an oriented spherical curve $\dot{P}^\epsilon$ with a base point obtained from $P$, it is easy to see that $\dxu(\dot{P}^\epsilon)$ ($= \dxb(\dot{P}^\epsilon)$) and $\dxl(\dot{P}^\epsilon)+\dxr(\dot{P}^\epsilon)$ do not depend on the orientation (see~Sec.\ \ref{sec2}).  

\noindent$\bullet$ ({\bf{Invariants under base point moves.}})
Consider Steps 1--3 to show the invariance under base point moves.  
Table~\ref{t1} gives the matrix 
$\left( 
\begin{matrix} 0 & 0 \\
-\beta_1 & \beta_2 \\
\beta_1 & -\beta_2 \\
0 & 0 \\
\end{matrix} 
\right)$.  Because the rank of this matrix is $1$, there exist two invariants under base point moves.   
From the matrix, it is easy to find the two invariants $\dxu(\dot{P}^\epsilon)$ ($=$ $\dxb(\dot{P}^\epsilon)$), $\dxl(\dot{P}^\epsilon)+\dxr(\dot{P}^\epsilon)$ under base point moves.
%, and $\dxb(P)$.  

In the rest of this paper, we may write $\dxu(P)$, $\dxl(P)+\dxr(P)$, and $\dxb(P)$ to represent $\dxu(\dot{P}^{\epsilon})$, $\dxl(\dot{P}^\epsilon)+\dxr(\dot{P}^\epsilon)$, and $\dxb(\dot{P}^\epsilon)$, respectively.   Note also that we have already proved that $\dxu(P)=\dxb(P)$ for any spherical curve $P$.  Thus, in the rest of this section, we exclude $\dxb(P)$.

\noindent$\bullet$ ({\bf{Invariants under RI and strong RI\!I\!I.}}) 
Similar to the above the discussion, Table~\ref{t6} gives the matrix $\left(
\begin{matrix}
1\\
1\\
%1
\end{matrix} \right)$ (see also Table~\ref{t3}).
%since we have Table~\ref{t6}.  
\begin{table}[h!]
\caption{Increment under a single $s3a$.}\label{t6}
\centering
\begin{tabular}{|c|c|} \hline
$\dxu$ & $1$ \\ \hline
$\dxl+\dxr$ & $1$ \\ \hline
%$\dxb$ & $1$ \\ \hline
\end{tabular}
\end{table}
The rank of this matrix is $1$.  Thus, there exists one invariant.  From the matrix, it is clear that the invariant is $\dxu(P)$ $-$ $\dxl(P)$ $-$ $\dxr(P)$. 

\noindent ({\bf{Invariants under RI and strong RI\!I.}})  Table~\ref{t7} gives the matrix 
$\left(\begin{matrix}
m \\
2m \\
%m
\end{matrix}\right)$ (see also Table~\ref{t4}).
% since we have Table~\ref{t7}.  
\begin{table}[h!]
\caption{Increment under a single $s2a$.}\label{t7}
\centering
\begin{tabular}{|c|c|}\hline
$\dxu$ & $m$\\ \hline
$\dxl + \dxr$ & $2m$\\ \hline
%$\dxb$ & $m$ \\ \hline
\end{tabular}
\end{table}
The rank of this matrix is $1$ and thus, there exists one invariant under RI and strong RI\!I as a linear combination of these based arrow diagrams.  From the matrix, we obtain the invariant $2 \dxu(P)-\dxl(P)-\dxr(P)$
%and $\dxu(P)-\dxb(P)$ 
 under RI and strong RI\!I.

\noindent ({\bf{Invariant under RI and weak RI\!I\!I.}})  Table~\ref{t10} gives the matrix 
$\left( \begin{matrix}
0 \\
1 \\
%0
\end{matrix} \right)$ (see also Table~\ref{t5}).   
\begin{table}[h!]
\caption{Increment under a single $w3a$.}\label{t10}
\centering
\begin{tabular}{|c|c|} \hline
\dxu & 0 \\ \hline
\dxl + \dxr & 1 \\ \hline
%\dxb & 0 \\ \hline
\end{tabular}
\end{table}  
The rank of this matrix is $1$.  Thus, there exists one invariant under RI and weak RI\!I\!I.  The matrix gives the invariant $\dxu(P)$.
% and $\dxb(P)$.  

%Because $\dxu(P)=\dxb(P)$ for an arbitrary $P$, we have a non-trivial invariant $\dxu(P)$.  

\noindent ({\bf{Invariant under RI and weak RI\!I.}})  One would expect to give some invariants under RI and weak RI\!I in the same way mentioned above.  Although we can find an invariant under RI and weak RI\!I in the same way, this invariant is trivial (i.e., all the values of the invariant are $0$ for all spherical curves).  %See the proof of Theorem~\ref{thm1} (the discussion just before that of additivity).   

Denote weak RI\!I increasing double points by $w2a$.  Because $\dxu(P)$ and $\dxl(P)$ $+$ $\dxr(P)$ are independent of the position of the base point, it is sufficient to consider only Fig.\ \ref{f17}.  
\begin{figure}[h!]
\centering
\includegraphics[width=5cm]{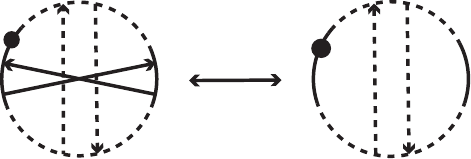}
\caption{Difference under a single weak RI\!I.}\label{f17}
\end{figure}
%See Fig.\ \ref{f17}.  
If we set the number of dotted arrows from the upper to lower part of the based arrow diagram as $m$, then the arrows from the lower to upper part of the based arrow diagram should be $m-1$ (cf.~Lemma~\ref{lemA}).  
%%question%%
Thus, we have the matrix 
$\left(\begin{matrix}
m \\
2m-1 \\
%m
\end{matrix}
\right)$ from Table~\ref{t9}.  
\begin{table}[h!]
\caption{Increment under a single $w2a$.}\label{t9}
\centering
\begin{tabular}{|c|c|} \hline
\dxu & $m$ \\ \hline
\dxl+\dxr & $2m-1$ \\ \hline
%\dxr & m \\ \hline 
%\dxb & $m$ \\ \hline
\end{tabular}
\end{table}
For each $m$, the rank of the matrix is $1$.  Then the dimension of the kernel is $1$ for each $m$.  If we vary $m$, then the intersection of the kernels has dimension $0$.  
Thus, there are no non-trivial invariants in this case.    

\section{Relation between Arnold invariants and Theorem~\ref{thm1}}\label{secArnold}
Additive integer-valued invariants $J^+$ (an invariant under strong RI\!I and RI\!I\!I), $J^-$ (an invariant under weak RI\!I and RI\!I\!I), and $St$ (an invariant under RI\!I) were defined by Arnold \cite{arnold}.  Polyak's Gauss diagram formulae \cite[Theorem~1]{polyak} can be used to yield Lemma~\ref{lemmaB}. (Note that there is a typo in Polyak's Theorem; thus, replace the coefficient $(-\frac{1}{2}, \frac{1}{2}, \frac{3}{2})$ with $(-\frac{1}{2}, \frac{1}{2}, \frac{1}{2})$ in the $St$-formula \cite[Theorem~1]{polyak}.)  
\begin{lemma}\label{lemmaB}
In the set of any linear combinations of three invariants $J^+$, $J^-$, and $St$, up to multiplying a constant, only the linear combination $J^+/2 + St$ is invariant under RI for spherical curves.  
\end{lemma}
For a spherical curve $P$, it is well-known that $J^+/2 + St$ does not depend on an orientation of $P$.  By the definitions of $J^+$ and $St$, we have the differences under each Reidemeister move, as shown in Table~\ref{arnoldTable}.
\begin{table}[h!]
\caption{Differences of $J^+/2 + St$ under Reidemeister moves.}\label{arnoldTable}
\centering
\begin{tabular}{|c|c|c|c|c|c|}\hline
&$RI$&$s2a$&$w2a$&$s3a$&$w3a$ \\ \hline
$J^+/2 + St$ &$0$&$0$&$+1$&$+1$&$-1$ \\ \hline
\end{tabular}
\end{table}    
\begin{theorem}
Let $P$ be an arbitrary spherical curve.  
$J^+/2(P) + St(P)$ is equal to $\dxu(P)$ $-$ $\dxl(P)$ $-$ $\dxr(P)$ $+$ $\dxb(P)$.  Invariants $\dxl(P)$ $+$ $\dxr(P)$ $-$ $\dxb(P)$ and $\dxu(P)$ cannot be presented as a linear combination of Arnold invariants $J^+$, $J^-$, and $St$.  Moreover, any two invariants among the invariants $J^+/2 + St$, $\dxb(P)$ $-$ $\dxl(P)$ $-$ $\dxr(P)$, and $\dxu(P)$ are mutually independent.  
\end{theorem}
\begin{proof}
$\dxu(P)-\dxl(P)-\dxr(P)+\dxb(P)$ changes by $0$ under RI, by $0$ under $s2a$ (Table~\ref{t7}), by $+1$ under $w2a$ (Table~\ref{t9}), by $+1$ under $s3a$ (Table~\ref{t6}), and by $-1$ under $w3a$ (Table~\ref{t10}).  For a simple closed curve $\bigcircle$, $J^+(\bigcircle)/2 + St(\bigcircle)=0$ and $\dxu(\bigcircle)-\dxl(\bigcircle)-\dxr(\bigcircle)+\dxb(\bigcircle)=0$; hence, we have the first claim.  

Now, we prove the second claim.  $\dxl(P)+\dxr(P)-\dxb(P)$ and $\dxu(P)$ are invariant under RI and are non-trivial.  Thus, Lemma~\ref{lemmaB} implies the second claim.  

Finally, let $3_1$, $4_1$, and $6_2$ be the spherical curves defined in Table~\ref{table5}.  
We have
\begin{align*}
&J^+(3_1)/2 + St(3_1) = 1, J^+(4_1)/2 + St(4_1) = 0, J^+(6_2)/2 + St(6_2) = 0,\\
&\dxl(3_1) + \dxr(3_1) - \dxb(3_1) = 0, \dxl(4_1) + \dxr(4_1) - \dxb(4_1) = 1, \\
&\dxl(6_2) + \dxr(6_2) - \dxb(6_2)=2, {\rm{~and~}}\dxu(3_1) = \dxu(4_1) = 1.  
\end{align*}
Therefore, we have the third claim.    
\end{proof}
\begin{corollary}
Let $P$ be a spherical curve and let $\x(P)$ be $\dxu(P)$ $+$ $\dxl(P)$ $+$ $\dxr(P)$ $+$ $\dxb(P)$.  
If $J^+(P)/2$ $+$ $St(P)$ $=$ $0$, $\x(P)$ $=$ $4 \dxu(P)$; in particular, $\x(P)$ $\equiv$ $0$ $($mod $4$$)$.  
\end{corollary}
\begin{proof}
If $\dxu(P)+\dxb(P)$ $=$ $\dxl(P)+\dxr(P)$, we have
\begin{equation*}
%\begin{split}
\x(P) = \dxu(P) + \dxl(P) + \dxr(P) + \dxb(P)  
%\\
%&= 2 (\dxu(P)+\dxb(P)) \\
%&
= 4 \dxu(P).
%\end{split}
\end{equation*}
\end{proof}
\begin{remark}
The invariant $\dxu(P)$ is non-trivial (cf.~Table~\ref{table5}).  
\end{remark}

\section{An invariant under RI and weak RI\!I}\label{sec4}
An integer-valued invariant under RI and weak RI\!I is not covered by Theorem~\ref{thm1}.    
Let us consider a {\it{Seifert resolution}}, as shown in Fig.\ \ref{sr}, for every double point of a spherical curve $P$.
\begin{figure}[h!]
\centering
\includegraphics[width=3cm]{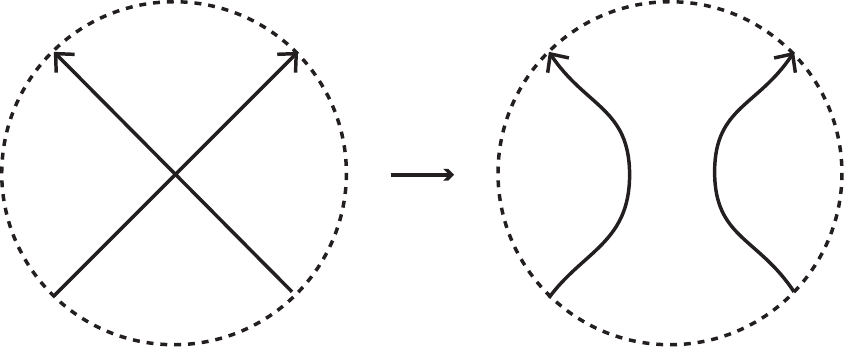}
\caption{Seifert resolution.}\label{sr}
\end{figure}
After all double points are resolved, we have an arrangement of finite number of circles.  The arrangement of circles on a sphere is denoted by $S(P)$.  It is known that $S(P)$ is invariant under weak RI\!I and weak RI\!I\!I.  The number of circles in $S(P)$ is denoted by $s(P)$.
% and is called the {\it{Seifert circle number}}.  
\begin{theorem}\label{thm_w12}
Let $P$ be an arbitrary spherical curve and $c(P)$ be the number of double points of $P$.
\[2 \dxu(P) - 2 \dxl(P) - 2 \dxr(P) + 2 \dxb(P) + s(P) - c(P)\]
is invariant under RI and weak RI\!I and changes by $-2$ under $w3a$.  
Moreover, 
\[s(P)-c(P)\]
is invariant under RI and weak RI\!I\!I.  These invariants are non-trivial.    
\end{theorem}
\begin{proof}
Let $1a$ be a single RI increasing double points.  
Table~\ref{t11} shows the increments for each local move and the invariances.
\begin{table}[h!]
\caption{Differences under $1a$, $w2a$, and $w3a$.}\label{t11}
\centering
\begin{tabular}{|c|c|c|c|} \hline
& $1a$ & $w2a$ & $w3a$  \\ \hline
$2\dxu(P)-2\dxl(P)-2\dxr(P)+2\dxb(P)$ & $0$ & $+2$ & $-2$  \\ \hline
$s(P)$ & $+1$ & $0$ & $0$ \\ \hline
$c(P)$ & $+1$ & $+2$ & $0$ \\ \hline
\end{tabular}
\end{table}
Non-trivialities are provided in Tables~\ref{tt1} and \ref{tt3}. 
\end{proof}

We have Proposition~\ref{prop1}.  
\begin{proposition}\label{prop1}
Let $P \sharp P'$ be the connected sum of the spherical curves $P$ and $P'$.  
$\kappa(P)=s(P)-c(P)$ and $\mu(P)$ $=$ $2\dxb(P)$ $-2\dxl(P)$ $-2\dxr(P)$ $+$ $2\dxu(P)$ $+$ $s(P)$ $-$ $c(P)$ satisfy the following properties.  
\begin{enumerate}
%\item $|S(P)|$ is an odd integer.  
%\item $s(P \sharp P')$ $=$ $s(P)$ $+$ $s(P')$ $-1$.
\item $\kappa(P \sharp_{(p_1, p_2)} P')$ $=$ $\kappa(P)$ $+$ $\kappa(P')$ $-1$. 
\item $\mu(P \sharp_{(p_1, p_2)} P')$ $=$ $\mu(P)$ $+$ $\mu(P')$ $-1$.   
\end{enumerate}
\end{proposition}

\section{Tables}\label{secTable}
In this section, we obtain five tables of prime spherical curves without $1$-gons up to seven double points and the trivial spherical curve (i.e., the simple closed curve).  A spherical curve $P$ is said to be {\it{prime}} if $P$ cannot be represented as a connected sum of two non-trivial spherical curves.  We define $1$-gon as the boundary with exactly one vertex of a disk.  In the tables, symbol $c_m$,  for positive integers $c$ and $m$, indicates the image of a projection of a knot in the Rolfsen table.  To list all prime spherical curves up to seven double points, it is sufficient to consider any flype for every $c_m$ (cf.~Tait flyping conjecture).  Thus, spherical curves $7_A$ (from $7_6$), $7_B$ (from $7_7$), and $7_C$ (from $7_5$) should be added by flypes.  In Tables~\ref{tt1}--\ref{tt4}, every line indicates that a finite sequence consisting of finitely many RIs and a single Reidemeister move $M$ ($\neq$ RI) has been found, where $M$ is weak RI\!I\!I for Table~\ref{tt1}, strong RI\!I\!I for Table~\ref{tt2}, weak RI\!I for Table~\ref{tt3}, and strong RI\!I for Table~\ref{tt4}.  We can show that any line requires at least a single $M$.  The dotted arc in Table~\ref{tt1} (resp.~\ref{tt2}) indicates that there exists a finite sequence consisting of two RIs and two weak RI\!I\!Is (resp.~three strong RI\!I\!Is) \cite{ildt}.    
\begin{table}[h!]
\caption{Values of the invariants under RI and weak RI\!I\!I by Theorem~\ref{thm1} and Theorem~\ref{thm_w12} for prime spherical curves without $1$-gons up to seven double points.  Symbol ($k$) with integer $k$ indicates the value of the invariant of Theorem~\ref{thm_w12}.}\label{tt1}
\begin{center}
\includegraphics[width=12cm]{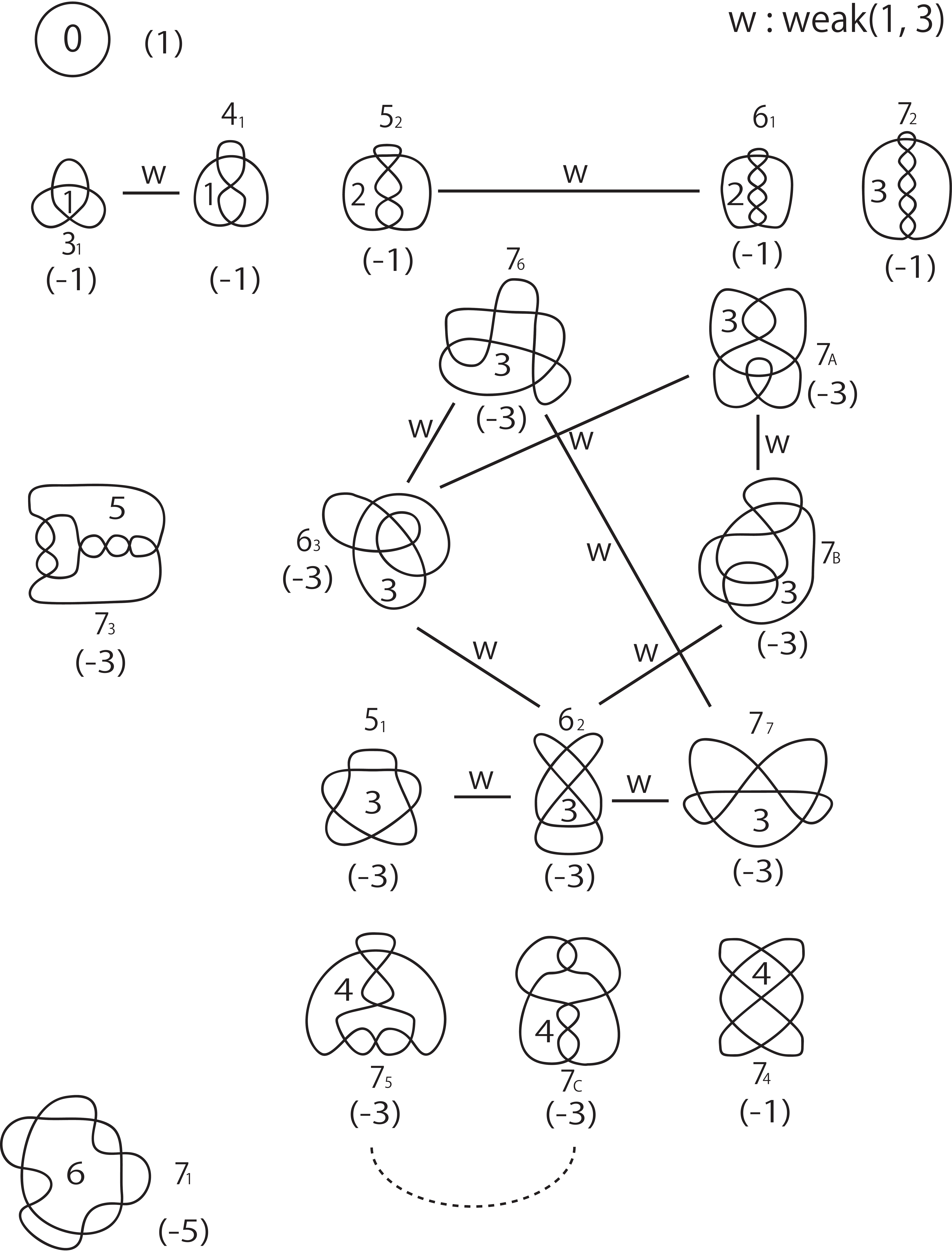}
\end{center}
\end{table}
\begin{table}[h!]
\caption{Values of the invariant under RI and strong RI\!I\!I by Theorem~\ref{thm1} for spherical curves without $1$-gons up to seven double points.  }\label{tt2}
\begin{center}
\includegraphics[width=12cm]{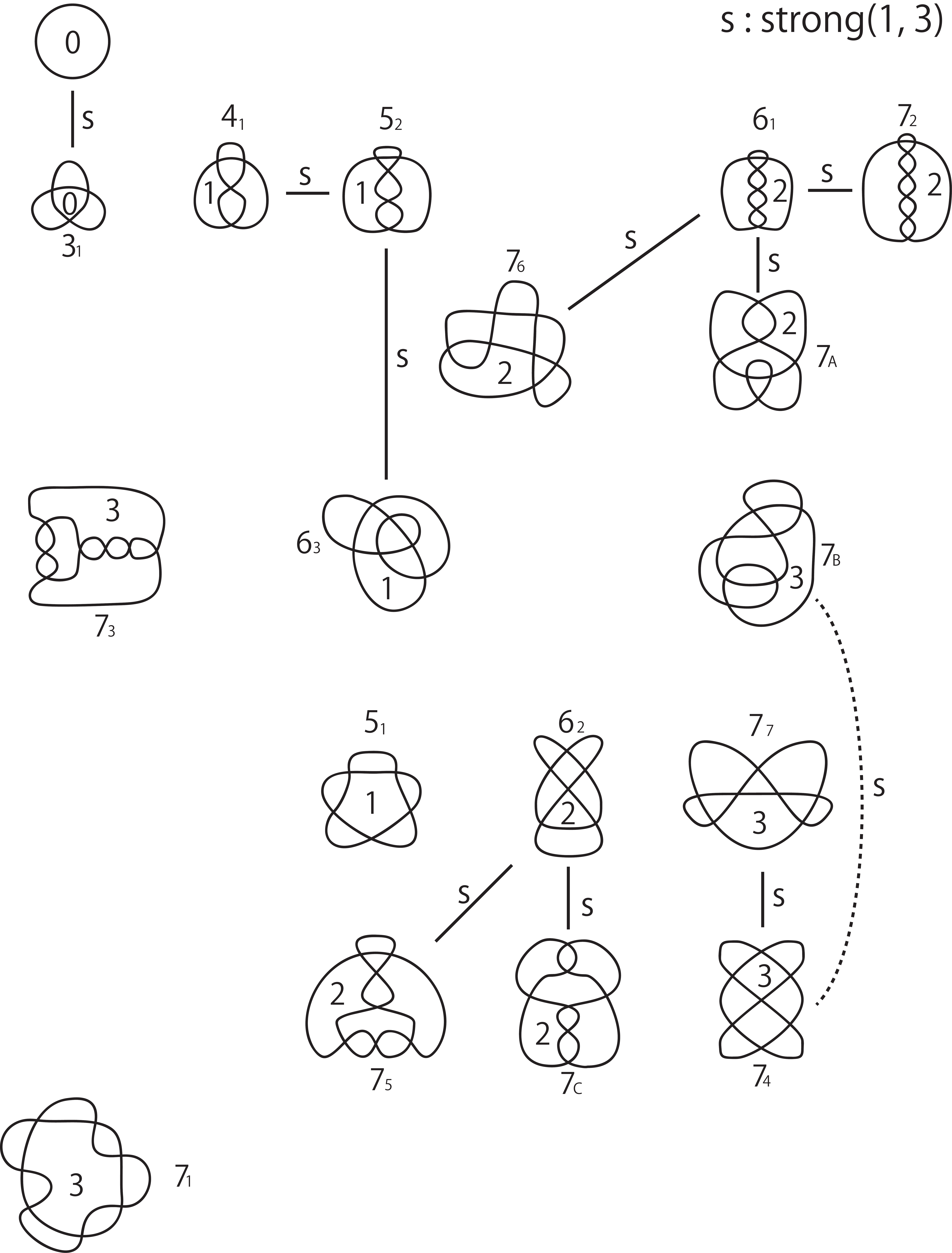}
\end{center}
\end{table}
\begin{table}[h!]
\caption{Values of the invariant under RI and weak RI\!I by Theorem~\ref{thm_w12} for prime spherical curves without $1$-gons up to seven double points.}\label{tt3}
\begin{center}
\includegraphics[width=12cm]{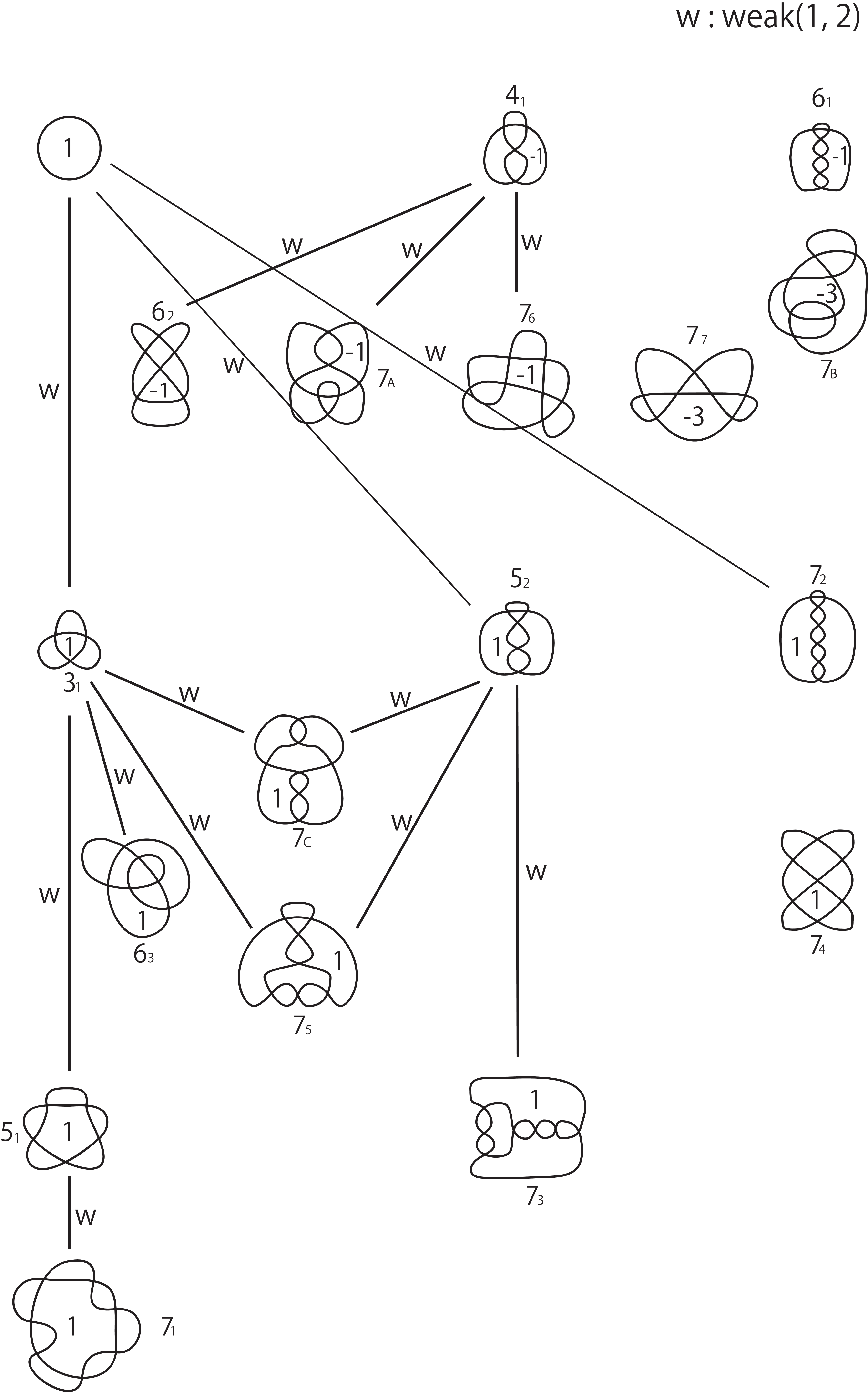}
\end{center}
\end{table}
\begin{table}[h!]
\caption{Values of the invariant under RI and strong RI\!I by Theorem~\ref{thm1} for prime spherical curves without $1$-gons up to seven double points.}\label{tt4}
\centering
\includegraphics[width=12cm]{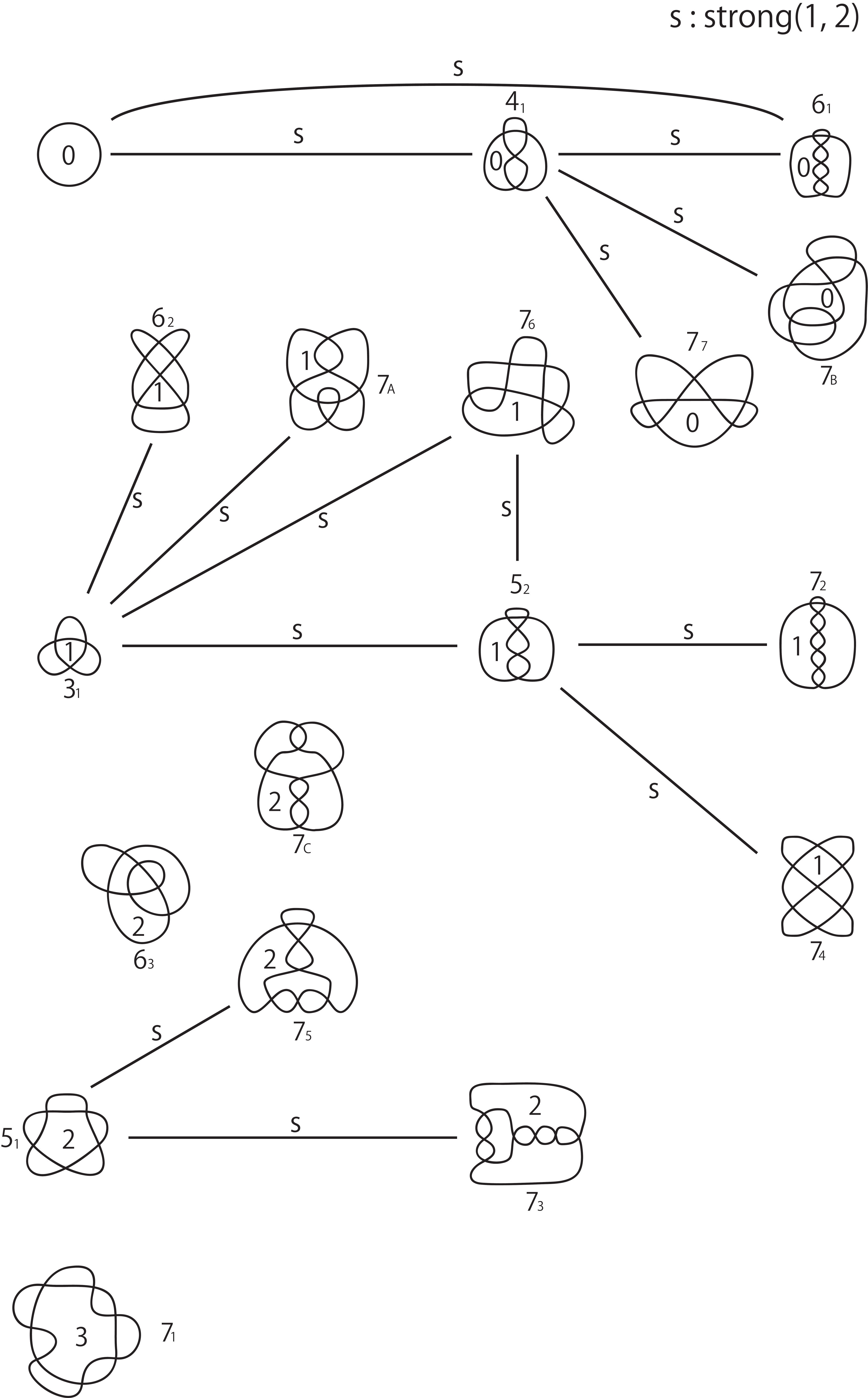}
\end{table}
\begin{table}[h!]
\caption{Values of base-point-free functions for a prime spherical curve $P$ without $1$-gons up to seven double points. }\label{table5}
\centering
\includegraphics[width=6.5cm]{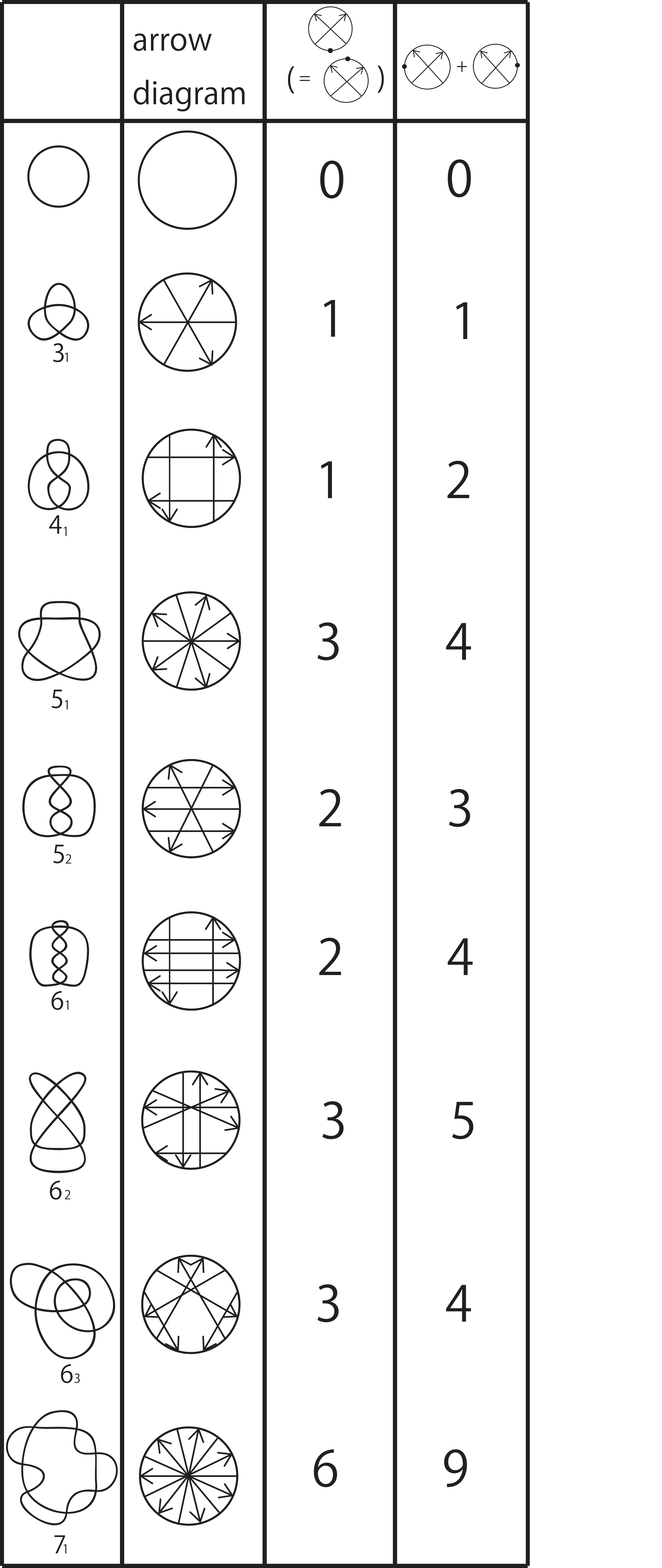}
\includegraphics[width=5.2cm]{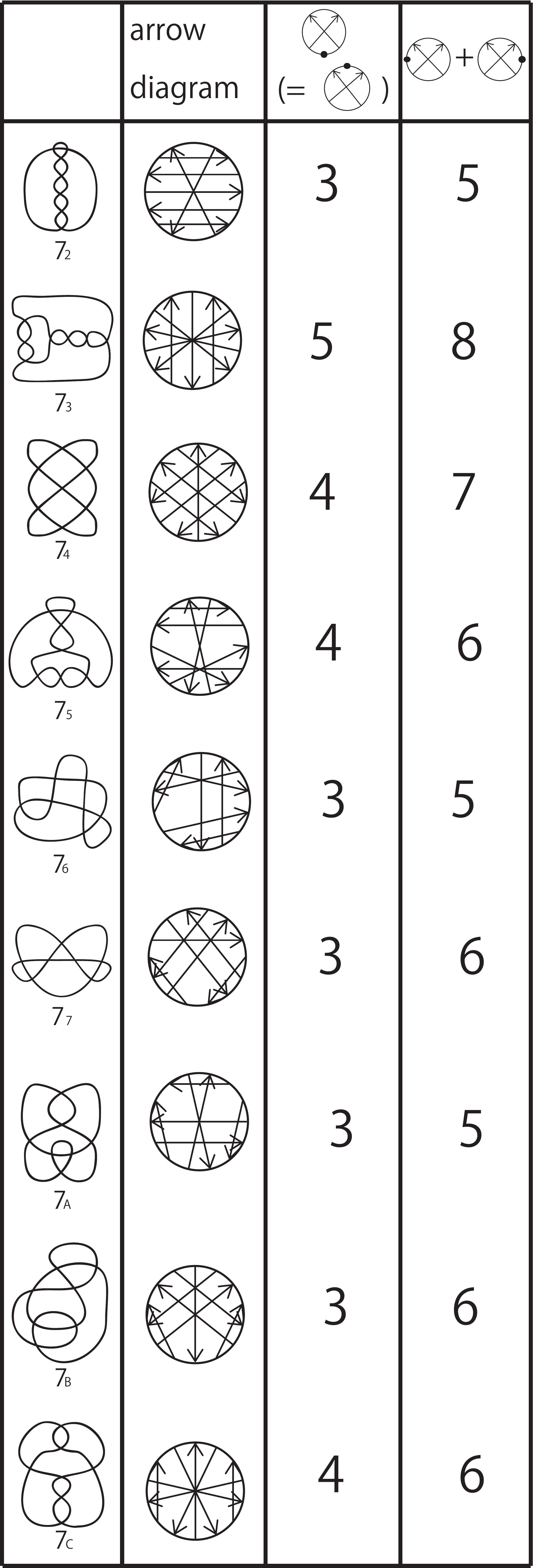}
\end{table}

\section*{Acknowledgements}
The author would like to thank Professor Kouki Taniyama for the fruitful discussions.  The author would also like to thank Mr.~Yusuke Takimura for the numerous discussions and providing the data to construct Tables~\ref{tt1}--\ref{table5}.


\clearpage

\begin{thebibliography}{99}
\bibitem{arnold}V.~I.~Arnold, Topological invariants of plane curves and caustics, Dean Jacqueline B.~Lewis Memorial Lectures presented at Rutgers University, New Brunswick, New Jersey.  University Lecture Series, {\bf{5}}.  
American Mathematical Society, Providence, RI, 1994.  
\bibitem{IT_some_ch} N.~Ito and Y.~Takimura, Sub-chord diagrams of knot projections, Houston J. Math. {\bf{41}} (2015), 701--725.    
\bibitem{ITT}N.~Ito, Y.~Takimura, and K.~Taniyama, Strong and weak (1, 3) homotopies on knot projections, Osaka J. Math. {\bf{52}} (2015), 617--646.   
\bibitem{ildt}N.~Ito, Y.~Takimura, and K.~Taniyama, Strong and weak (1, 3) homotopies on spherical curves and related topics, RIMS K\^{o}ky\^{u}roku, No.~{\bf{1960}} (2015), 101--106.   
\bibitem{polyak} M.~Polyak, Invariants of curves and fronts via Gauss diagrams, Topology {\bf{37}} (1998), 989--1009.
\bibitem{PV} M.~Polyak and O.~Viro, Gauss diagram formulas for Vassiliev invariants, Internat. Math. Res. Notices {\bf{1994}}, 445ff., approx. 8pp (electronic).  
\end{thebibliography}
\end{document}